\documentclass[11pt,a4]{article}
\usepackage{amssymb}
\usepackage{amsfonts}
\usepackage{amsthm}
\usepackage{amsbsy}
\usepackage{graphicx}
\usepackage{verbatim}
\usepackage{enumerate}
\usepackage{enumitem}
\usepackage{hyperref}
\def\BE{\begin{equation}}
\def\EE{\end{equation}}
\def\BEA{\begin{eqnarray}}
\def\EEA{\end{eqnarray}}

\def\NNL{\nonumber \\}
\def\NN{\nonumber}
\def\NewAND{\hspace{-0.15in} & & \hspace{-0.15in}}

\newtheorem{Theo}{Theorem}[section]
\newtheorem{Lem}{Lemma}[section]

\newtheorem{Rem}{Remark}[section]

\theoremstyle{Definition}

\providecommand{\MR}{\relax\ifhmode\unskip\space\fi MR }

\providecommand{\href}[2]{#2}

\begin{document}

\title{Hidden Markov Mixture Autoregressive Models:
Parameter Estimation}

\author{S.H.Alizadeh, S.Rezakhah\footnotemark[1]}
\footnotetext[1]{	
            Faculty of Mathematics and Computer Science
			,Amirkabir University of Technology, Tehran, Iran.
            Email:\texttt{rezakhah@aut.ac.ir,\, sasan\_alizadeh@aut.ac.ir} \par
			}
\date{}
\maketitle

\begin{abstract}
This report introduces a parsimonious structure for mixture of autoregressive models, where the weighting coefficients are determined through latent random variables as functions of all past observations. These variables follow a hidden Markov model. We modify EM and Baum-Welch algorithms to estimate the parameters of the model.
\end{abstract}


 \par

\section{Hidden Markov Mixture Autoregressive Model} \label{Sec HM-MAR}

Let $Y=\{Y_t\}_{t=0}^\infty$ be a sequence of {continuous}  random variables, {where}  $y_t$ {is} a realization of $Y_t$. {Also let} $\mathcal{F}_{t}=\sigma\{Y_s:s \leq t\}$ {represents} the sigma-field of all information up to time $t$, $F(y_t|\mathcal{F}_{t-1})$  the conditional distribution function of $Y_t$ given past information and  $\alpha_h^{(t)} \equiv \alpha_h^{(t)}(y_1,...,y_{t-1})$. {In addition}
 $\{Z_t\}_{t \geq p}$ {denotes} a hidden or latent process {which construct a} positive recurrent Markov chain on a finite set $E=\{1,2,...,K\}$, {with the} initial conditional probabilities
\BEA
   \boldsymbol{\rho}=(\rho_1,\cdots,\rho_K)^\prime,\,\,\, \rho_h=P(Z_{p}=h|y_0,\cdots,y_{p-1}) \qquad
   h=1,...,K, \label{Initial Probabilities}
\EEA
and transition probability matrix
\BEA
    P=\|\pi_{i,j}\|_{K \times K},
    \label{Transition Probability Matrix}
\EEA
in which
\BEA
    \pi_{i,j}=P(Z_t=j|Z_{t-1}=i),\qquad i,j \in \{1,...,K\}.
    \label{Transition Probabilities}
\EEA
{Also} invariant probability measure {is denoted by}
\BEA
    \boldsymbol{\mu}=(\alpha_1,...,\alpha_K)^\prime, \label{Mu}
\EEA
{where} $\alpha_j=\lim_{t \rightarrow \infty} P(Z_t=j)$.

We consider $\{Y_t\}_{t=0}^\infty$ to have a Hidden Markov-Mixture Autoregressive, HM-MAR($K,p$), model with $K$ normal distributions, {and} $p$ lagged observations in the AR processes, if the conditional distribution of $Y_t$ given $\mathcal{F}_{t-1}$ {follows}

\begin{enumerate}
    \item[i.]

        {For} $t=p$
         \BEA
            F(y_{p},Z_p=h|\mathcal{F}_{p-1}) =  \rho_h
                                   \Phi(\frac{y_p-a_{0,h}-a_{1,h}y_{p-1}-...-a_{{p},h}y_{0}}{\sigma_h}),
                                   \label{HM MAR p1}
         \EEA

    \item[ii.]

        {For} $t \geq p + 1$
        \BEA
            F(y_{t}|\mathcal{F}_{t-1})= \sum_{h=1}^K \alpha_h^{(t)}
            \Phi(\frac{y_t-a_{0,h}-a_{1,h}y_{t-1}-...-a_{{p},h}y_{t-{p}}}{\sigma_h}),
            \label{HM MAR}
        \EEA
\end{enumerate}
{where} $\alpha_h^{(t)}=P(Z_t=h|\mathcal{F}_{t-1})$ and $\Phi(.)$ is the standard normal distribution function.

In fact latent random variables $\{Z_t\}_{t=p+1}^\infty$ determine the contribution of distributions in the mixture model. Also conditioning on $Z_t$, $\{Y_t,t \in \mathbb{N}\}$ is $p$-tuple Markov{,} independent of $\{Z_{s},\, s \neq t \}$. {So by} conditioning on $\{Y_{t-1},\cdots,Y_{t-p}\}$ and $Z_t$, $Y_t$ {is independent of} $\{Y_{s}, \, s < t-p\}$ and $\{Z_{s},\, s \neq t \}$.

The novelty of HM-MAR model is that
the contribution of each distribution in the mixture structure is not of predefined fixed form.
{Although HM-MAR model uses all} past observations from $Y_0$ to $Y_{t-1}$ {but}
the hidden Markov assumption of the process $\{Z_t\}_{t \geq p}$, enables us to build a parsimonious model.

The MAR model \cite{WL} can be considered as a special case of such a HM-MAR model (\ref{HM MAR p1}-\ref{HM MAR}), in which the transition matrix $P$ of the process $\{Z_t\}_{t \geq p}$ has $K$ identical rows (i.e. $p(Z_t=i|Z_{t-1}=j)=\alpha_i$ for all $i,\,j=1,...,K${. Thus} $\{Z_t\}_{t=p+1}^\infty$ are independent and identically distributed) with $p(Z_t=i|Z_{t-1}=j)=\alpha_i$.

HM-MAR model will also lead to hidden Markov model in general state space where $p$ is considered to be zero in (\ref{HM MAR}) (i.e. $Y_t$ given $Z_t$, is independent of past observations).

\section{Estimation}
In this section, we discuss estimation of parameters of a HM-MAR$(K,p)$ model.
A new algorithm is proposed based on  modification of Baum-Welch \cite{macdonald}  and EM \cite{mclachlan} algorithms.
Baum welch algorithm was originally proposed in the context of Hidden Markov Models for parameter estimation
(For a comprehensive review see MacDonal and Zucchini \cite{macdonald}).
In HMM each observation just depends on a state of a hidden variable, however in HM-MAR, past observations have also effect on next time series observation. First we justify that the modification of Baum-Welch algorithm is correct and then modify the EM algorithm for the case where the latent variable follows a Hidden Markov process.

Let denote $\boldsymbol{A}_j=(a_{0,j},\cdots,a_{p,j})^\prime$ then $\theta=\{\boldsymbol{A}_j, \, \sigma_j, \, \rho_j, \, \pi_{mn}, \, m,n,j=1,\cdots,K\}$  constitutes the parameter set of HM-MAR model, which includes $\{K^2 + (p+2)K\}$ parameters. As $Y_t$ given $Z_t$ forms a $p$-tuple Markov in HM-MAR model, its conditional distribution can be written as
\BEA
    F(y_t|y_{0}...y_{t-1},z_{t})= \prod_{k=1}^K \Phi(\frac{y_t-\mathbf{Y}_{t-1}^\prime A_k}{\sigma_k})^{I(z_t=k)},
    \label{conditional prob}
\EEA
where $\mathbf{Y}_{t}=(1,y_{t},\cdots,y_{t-p+1})^\prime$,
also the conditional distribution $P(z_{t}|z_{t-1})$ is given by
\BEA
    P(Z_t=z_{t}|Z_{t-1}=z_{t-1})=\prod_j\prod_k \pi_{j,k}^{I(z_t=k)I(z_{t-1}=j)}. \label{hidden Markov Transition prob}
\EEA

\subsection{Extension of Baum-Welch Algorithm}
\begin{Lem} \label{Lemma Conditional Independence}
    Let $\{y_t\}_{t=0}^T$ be a set of time series observations and $\{Z_t\}$ be a set of correct predictor indexes, in ARSNN next time series observations just depends on the last correct predictor. That is for $t \leq k \leq T$
    \BEA
       F(y_{t+1},\cdots,y_K|y_1,\cdots,y_{t}, \NewAND \{Z_{s}\}_{s \in \mathbb{N},\,s
       \leq t})
       \NNL & = &
       F(y_{t+1},\cdots,y_K|y_1,\cdots,y_{t},Z_{t}) \label{PYtZt}
    \EEA
\end{Lem}
\begin{proof}
Considering the homogeneous hidden Markov structure assumption of $\{Z_t\}$ in HM-MAR model (\ref{HM MAR p1}-\ref{HM MAR}) and the assumption that  $y_t$ given we have information about the $Z_t$, just depends on $p$ lagged time series observations through \ref{conditional prob}, we use the method of induction to prove (\ref{PYtZt}). So for $k=t+1$ we have that
\BEA
    F(y_{t+1}|y_1,\cdots,y_{t}, \NewAND \{Z_{s}\}_{s \in \mathbb{N},\,s \leq t})
    \NNL &=& \sum_{j=1}^K F(y_{t+1},Z_{t+1}=j|y_1,\cdots,y_{t},\{Z_{s}\}_{s \in \mathbb{N},\,s \leq t})
    \NNL &=& \sum_{j=1}^K F(y_{t+1}|y_1,\cdots,y_{t},Z_{t+1}=j)P(Z_{t+1}=j|Z_{t}), \NN
\EEA
which is independent of $\{Z_{t-i},\, i \in \mathbb{N},i > 1 \}$.
Now assume that equation (\ref{PYtZt}) holds for $t+1<\ell<T$, that is
\BEA
F(y_{t+1},\cdots,y_\ell|y_1,\cdots,y_{t},\NewAND \hspace{-0.1in} \{Z_{s}\}_{s \in \mathbb{N},\,s \leq t})
\NNL & = &
F(y_{t+1},\cdots,y_\ell|y_1,\cdots,y_{t},Z_{t}) \label{Induction Asmpt}
\EEA
We show that (\ref{PYtZt}) is valid for $k=\ell+1$
\BEA
    && F(y_{t+1},\cdots,y_{\ell},y_{\ell+1}|y_1,\cdots,y_{t},\{Z_{s}\}_{s \in \mathbb{N},\,s \leq t}) =
    \NNL && \sum_{j=1}^K F(y_{\ell+1}|y_1,\cdots,y_{\ell},Z_{\ell+1}=j)
    P(Z_{\ell+1}|y_1,\cdots,y_{\ell},\{Z_{s}\}_{s \in \mathbb{N},\,s \leq t}) \times
    \NNL && F(y_{\ell}|y_1,\cdots,y_{t},\{Z_{s}\}_{s \in \mathbb{N},\,s \leq t})
    \NNL
    &=&
    \sum_{j=1}^K F(y_{\ell+1}|y_1,\cdots,y_{\ell},Z_{\ell+1}=j) P(Z_{\ell+1}|Z_{t})\times
    \NNL & &
    F(y_{\ell}|y_1,\cdots,y_{t},\{Z_{s}\}_{s \in \mathbb{N},\,s \leq t}), \NN
\EEA
which is independent of $\{Z_{t-i}\}_{i \geq 1}$ by the induction's assumption (\ref{Induction Asmpt}).
\end{proof}

\begin{Theo} \label{Theo alpha ht}
    Let for $t>p$
    \BEA
        && \alpha_t(h)=F(y_{p+1}...y_t,Z_t=h|y_1...y_p), \label{alpha th}
        \\ &&  \beta_t(h)=F(y_{t+1}...y_T|y_1...y_t,Z_t=h), \label{beta th}
    \EEA
    then $\alpha_t(h)$ and $\beta_t(h)$ can be calculated by Baum-welch forward backward recursions as
    \BEA
        \alpha_{t+1}(h) & = & \sum_m \pi_{m,h} \alpha_t(m)  \Phi(\frac{y_{t+1}-\mathbf{Y}_{t}^\prime A_h}{\sigma_h})
        \NNL
        \beta_t(h) & = & \sum_{j=1}^K  \pi_{h,j} \beta_{t+1}(j) \Phi(\frac{y_{t+1}-\mathbf{Y}_{t-1}^\prime A_k}{\sigma_j}).
    \EEA
    And the forward recursion starts with $\alpha_{p+1}(h)=\rho_h \Phi\{(y_{p+1}-\mathbf{Y}_{p}^\prime A_h)/{\sigma_h}\}$ and backward recursion starts at $\beta_T(h) =1$, in which $\Phi(.)$ is the standard normal distribution function.
\end{Theo}
\begin{proof}
$\alpha_t(h)$ in equation (\ref{alpha th}) can be written as
\BEA
    \alpha_{t+1}(h) & = & \sum_m F(y_{p+1}...y_{t+1},Z_t=m,Z_{t+1}=h|y_1...y_p) \nonumber \\
                    & = & \sum_m p(Z_{t+1}=h|Z_t=m,y_1...y_t) \times F(y_{t+1}|y_1...y_t,Z_t=m,Z_{t+1}=h)
                     \nonumber \\ &\times& F(y_{p+1}...y_t,Z_t=m|y_1...y_p)
                     \nonumber \\ & = & \sum_m \pi_{m,h} \alpha_t(m)  \Phi(\frac{y_t-\mathbf{Y}_{t-1}^\prime A_h}{\sigma_h})
\EEA
Also by lemma \ref{Lemma Conditional Independence}, for $\beta_t(h)$ in equation (\ref{beta th}) we have
\BEA
    \beta_t(h) &=& \sum_{j=1}^K F(Z_{t+1}=j,y_{t+1}...y_T|y_1...y_t,Z_t=h)
    \NNL &=& \sum_{j=1}^K F(y_{t+2}...y_T|y_1...y_t,y_{t+1},Z_t=h,Z_{t+1}=j)
    \times
    \NNL &&
    F(y_{t+1}|y_1...y_t,Z_{t}=h,Z_{t+1}=j)
    p(Z_{t+1}=j|y_1...y_t,Z_{t}=h)
    \NNL & = & \sum_{j=1}^K  \pi_{h,j} \beta_{t+1}(j) \Phi(\frac{y_{t+1}-\mathbf{Y}_{t-1}^\prime A_j}{\sigma_j})
\EEA
\end{proof}

\subsection{Modification of EM Algorithm}
The EM algorithm is used for maximization of completed data log-likelihood. By completed data we mean that the set of time series observations $\{y_t\}_{t=1}^T$  augmented with the latent set of correct predictor indicators $\{z_t\}_{t=p+1}^T$ (i.e. $\{\{y_t\}_{t=1}^T,\{z_t\}_{t=p+1}^T\}$). So this log-likelihood, by the method of iterative conditioning, can be represented as
\BEA
     \ell^*(\theta) & = &  \log F(y_{p+1}...y_T,z_{p+1}...z_{T}|y_1...y_p)
                \NNL & = &
                \sum_{t=p+1}^{T} \log(F(y_t|y_{t-1},\cdots,y_{0},z_t)) +  \sum_{t=p+2}^{T} \log(P(z_t|z_{t-1},\cdots,Z{p},y_{t-1},\cdots,y_{0})) +
                \NNL &&  \log{P(z_{p+1}|y_1,\cdots,y_p)}
                \NNL &=&
                \sum_{t=p+1}^T \sum_k I(z_t=k) \log \Phi(\frac{y_t-\mathbf{Y}_{t-1}^\prime A_k}{\sigma_k}) + \nonumber \\ &&
                    \sum_{t=p+2}^T \sum_k \sum_j I(z_t=k)I(z_{t-1}=j) \log \pi_{j,k} + \nonumber \\ &&
                    \sum_k I(z_{p+1}=k) \log \rho_k, \NN
\EEA
where the last equality holds by (\ref{conditional prob}) and the Markov property of $\{Z_t\}$ with transition probabilities in (\ref{hidden Markov Transition prob}). It is clear that $\sum_{t=p+2}^T I(z_t=k)I(z_{t-1}=j) $ is equal to the number of transitions from state $j$ to state $k$.
At the E-step, the algorithm computes the conditional expected value of each $I(z_t=k)$ and $I(z_t=k)I(z_{t-1}=j)$ given the observed data.
\BEA
    &&    E[\ell^*(\theta)|y_1,\cdots,y_T] =
                    \sum_{t=p+2}^T \sum_k \sum_j P(z_t=k,z_{t-1}=j|y_1...y_T) \log \pi_{j,k} + \nonumber \\ &&
                     \sum_{t=p+1}^T \sum_k P(z_t=k|y_1...y_T) \{-\log(\sqrt{2\pi})-\log(\sigma_k)-\frac{(y_t-\mathbf{Y}_{t-1}^\prime A_k)^2}{2\sigma_k^2}) \}
                     \nonumber \\ &&  +
                    \sum_k P(z_{p+1}=k|y_1...y_T) \log \rho_k. \label{Eell*}
\EEA
Last equation holds by linear property of expectation and since $\Phi\{(y_t-\mathbf{Y}_{t-1}^\prime A_k){\sigma_k}\}$ is measurable with respect to $\sigma\{Y_1,...,Y_T\}$. Also  $E(I(z_t=k)|y_1...y_T)=P(z_t=k|y_1...y_T)$ and $E(I(z_t=k)I(z_{t-1}=j|y_1...y_T)=P(z_t=k,z_{t-1}=j|y_1...y_T)$. These posterior probabilities can be obtained by the following lemma

\begin{Lem} \label{Lem Pzt}
$P(Z_t=h|Y_1,\cdots,Y_T)$ and $P(Z_t=j,Z_{t-1}=i|Y_1,\cdots,Y_T)$ in equation (\ref{Eell*}) can be calculated as
\BEA
    && P(Z_t=h|Y_1,\cdots,Y_T)  =   \frac{\alpha_{t}(h)\beta_t(h)}{F(Y_{p+1},\cdots,Y_T|Y_{1},\cdots,Y_p)}\, ,
    \NNL && P(Z_t=j,Z_{t-1}=i|Y_1,\cdots,Y_T)  =   \frac{\beta_t(j) \pi_{ij} \alpha_{t-1}(i)}{F(Y_{p+1},\cdots,Y_T|Y_{1},\cdots,Y_p)} \times
    \NNL &&
    \Phi(\frac{y_{t}-\mathbf{Y}_{t-1}^\prime A_k}{\sigma_j})
    \NN
\EEA
in which $F(Y_{p+1},\cdots,Y_T|Y_{1},\cdots,Y_p)=\sum_{j=1}^K \alpha_T(j)$ and $\{\alpha_t(.),\beta_t(.)\}_{t=p+1}^T$ are calculated by theorem \ref{Theo alpha ht}.
\end{Lem}

\begin{proof} Using equations (\ref{alpha th}) and (\ref{beta th}) we have
\BEA
    P(Z_t=h|Y_1,\cdots,Y_T) & = & \frac{F(Z_t=h,Y_1,\cdots,Y_T)}{F(Y_1,\cdots,Y_T)}
            \NNL &=& F(Z_t=h,Y_{p+1},\cdots,Y_t|Y_{1},\cdots,Y_p) \times
            \NNL & &  F(Y_{t+1},\cdots,Y_T|Z_t=h,Y_1,\cdots,Y_t) \times \frac{F(Y_1,\cdots,Y_p)}{F(Y_1,\cdots,Y_T)}
            \NNL &=& \frac{\alpha_{t}(h)\beta_t(h)}{F(Y_{p+1},\cdots,Y_T|Y_{1},\cdots,Y_p)}\, ,
\EEA
in which
$$F(Y_{p+1},\cdots,Y_T|Y_{1},\cdots,Y_p)= \sum_{j=1}^K F(Y_{p+1},\cdots,Y_T,Z_t=j|Y_{1},\cdots,Y_p) = \sum_{j=1}^K \alpha_T(j)$$
and by lemma \ref{Lemma Conditional Independence}, (\ref{conditional prob}) and Markov property of $\{Z_t\}$ we have that
\BEA
    && P(Z_t=j,Z_{t-1}=i|Y_1,\cdots,Y_T)  =
    \NNL & = &
    F(y_{t+1},\cdots,y_T|y_1,\cdots,y_t,z_t,z_{t-1})
    F(y_t|y_1,\cdots,y_{t-1},z_t,z_{t-1}) \times
    \NNL &&
    \frac{P(z_t|y_1,\cdots,y_{t-1},z_{t-1})
          F(y_{p+1},\cdots,y_{t-1},z_{t-1}|y_{1},\cdots,y_{p})}
        {F(y_1,\cdots,y_T)}
     \NNL &=&
     \frac{\beta_t(j) \pi_{ij} \alpha_{t-1}(i)}{F(Y_{p+1},\cdots,Y_T|Y_{1},\cdots,Y_p)}
    \Phi(\frac{y_{t}-\mathbf{Y}_{t-1}^\prime A_k}{\sigma_j})
    \NN
\EEA
\end{proof}

In the M-step, roots of equation $\partial E[\ell^*(\theta)|y_1,\cdots,y_T] / \partial \theta_i =0, \, \theta_i \in \theta$, are calculated

\begin{Theo}
Let $\widetilde{\mathbf{Y}}=(\mathbf{Y}_p,\cdots,\mathbf{Y}_{T-1})$, $\bar{\mathbf{Y}}=(y_{p+1},\cdots,y_T)^\prime$ and $\mathbf{P}_k= diag(P(Z_{p+1}=k|y_1...y_T),\cdots,P(Z_{T}=k|y_1...y_T))$, then maximum likelihood estimate of the parameters HM-MAR are given by
\BEA
    && \hspace{-0.5in} \hat{A}_k = (\widetilde{\mathbf{Y}}\mathbf{P}_k\widetilde{\mathbf{Y}}^\prime)^{-1}\widetilde{\mathbf{Y}}\mathbf{P}_k\bar{\mathbf{Y}}
    \\ &&
    \hspace{-0.5in} \hat{\sigma}_k^2 = \frac{
                            \{
                             \bar{\mathbf{Y}}^\prime \mathbf{P}_k(\mathbf{I}-\widetilde{\mathbf{Y}}^\prime
                                (\widetilde{\mathbf{Y}}\mathbf{P}_k\widetilde{\mathbf{Y}}^\prime)^{-1})
                                \bar{\mathbf{Y}}\widetilde{\mathbf{Y}}\mathbf{P}_k
                                -
                                2 \bar{\mathbf{Y}}^\prime \mathbf{P}_k \widetilde{\mathbf{Y}}^\prime
                                (\widetilde{\mathbf{Y}}\mathbf{P}_k\widetilde{\mathbf{Y}}^\prime)^{-1}\widetilde{\mathbf{Y}}\mathbf{P}_k\bar{\mathbf{Y}}
                             \}
                            }
                 {tr(\mathbf{P}_k)}  \label{sigma hat}
    \\ && \hspace{-0.5in}
    \hat{\pi}_{j,i} = \frac{\sum_{t=p+2}^T P(Z_t=i,Z_{t-1}=j|y_1,\cdots,y_T)}{\sum_{t=p+2}^T P(Z_{t-1}=j|y_1,\cdots,y_T)}
    \\ && \hspace{-0.5in}
    \hat{\rho}_j = \frac{\sum_{t=P+1}^T P(Z_t=j|Y_1,\cdots,Y_T)}{T-P}. \label{rho hat}
\EEA
\end{Theo}
\begin{proof}
calculating $\partial E[\ell^*(\theta)|y_1,\cdots,y_T] / \partial \mathbf{\phi}_k =0$, we obtain
\BEA
    && \sum_{t=p+1}^T P(z_t=k|y_1...y_T)\mathbf{Y}_{t-1}(y_t-\mathbf{Y}_{t-1}^\prime A_k) =0
    \NNL
    && \Rightarrow
    \widetilde{\mathbf{Y}}\mathbf{P}_k\bar{\mathbf{Y}}-\widetilde{\mathbf{Y}}\mathbf{P}_k\widetilde{\mathbf{Y}}^\prime A_k = 0
    \NNL
    && \Rightarrow
    \hat{A}_k = (\widetilde{\mathbf{Y}}\mathbf{P}_k\widetilde{\mathbf{Y}}^\prime)^{-1}\widetilde{\mathbf{Y}}\mathbf{P}_k\bar{\mathbf{Y}}
    \label{PhiHat}
\EEA
calculating $\partial E[\ell^*(\theta)|y_1,\cdots,y_T] / \partial \sigma_k =0$, we obtain
\BEA
    && \sum_{t=p+1}^T P(z_t=k|y_1...y_T)(-\frac{1}{\sigma_k}+\frac{(y_t-\mathbf{Y}_{t-1}^\prime A_k)^2}{\sigma_k^3}) =0
    \NNL
    && \Rightarrow   tr(\mathbf{P}_k) \sigma_k^2  =
        (\bar{\mathbf{Y}} - \widetilde{\mathbf{Y}}^\prime A_k)^\prime
        \mathbf{P}_k
        (\bar{\mathbf{Y}} - \widetilde{\mathbf{Y}}^\prime A_k)
    \NNL
    && = \bar{\mathbf{Y}}^\prime \mathbf{P}_k \bar{\mathbf{Y}} -
          2 \bar{\mathbf{Y}}^\prime \mathbf{P}_k \widetilde{\mathbf{Y}}^\prime A_k +
         A_k^\prime  \widetilde{\mathbf{Y}} \mathbf{P}_k \widetilde{\mathbf{Y}}^\prime A_k
\EEA
Since $(\bar{\mathbf{Y}}^\prime \mathbf{P}_k \widetilde{\mathbf{Y}}^\prime A_k )^\prime =
 A_k^\prime  \widetilde{\mathbf{Y}} \mathbf{P}_k \bar{\mathbf{Y}}$.
Replacing $\hat{A}_k$ from equation (\ref{PhiHat}), we obtain equation (\ref{sigma hat}) for $\hat{\sigma}_k^2$.

Since for each $j=1,\cdots,K$ in the transition matrix $P$ of Markov process $Z_t$, $\sum_{i=1}^K \pi_{j,i} = 1$ thus
\BEA
    \pi_{j,K} = 1 - \sum_{i=1}^{K-1} \pi_{j,i} \label{PiJK}.
\EEA
Calculating the roots of equation $\partial E[\ell^*(\theta)|y_1,\cdots,y_T] / \partial \pi_{j,i} = 0$, by equation (\ref{PiJK}), we have
\BEA
    && \pi_{j,i} = \pi_{j,K} \frac{\sum_{t=p+2}^T P(Z_t=i,Z_{t-1}=j|y_1,\cdots,y_T)}{\sum_{t=p+2}^T P(Z_t=K,Z_{t-1}=j|y_1,\cdots,y_T)}
    \NNL && \Rightarrow \hat{\pi}_{j,i} = \frac{\sum_{t=p+2}^T P(Z_t=i,Z_{t-1}=j|y_1,\cdots,y_T)}{\sum_{t=p+2}^T P(Z_{t-1}=j|y_1,\cdots,y_T)}
\EEA
In a similar way we obtain equation (\ref{rho hat}) for $\hat{\rho}_j$ for $j=1,\cdots,K$.
\end{proof}

\subsection{Learning}
A brief summary of HM-MAR(K,P) parameter estimation algorithm is as follows:
\newpage
\begin{enumerate}
\item  For t=1 to T do
    $$\mathbf{Y}_{t}=(y_{t},\cdots,y_{t-p+1})^\prime$$
\item Let
  $$
    \hspace{-1in}
    \begin{array}{l}
        \widetilde{\mathbf{Y}}=(\mathbf{Y}_p,\cdots,\mathbf{Y}_{T-1}) \\
        \bar{\mathbf{Y}}=(y_{p+1},\cdots,y_T)^\prime
     \end{array}
  $$
\item For h=1 to K do
    $$ \hspace{-0.9in}
    \begin{array}{l}
            A_h=(w_1^{1,h},\cdots,w_p^{1,h})^\prime
        \end{array}$$
\item Let
    $$ \hspace{-0.9in}
    \begin{array}{l}
            \rho_h=P(Z_{p+1}=h|y_1,\cdots,y_{p}) \\
             \theta=\{A_j, \, \sigma_j, \, \rho_j, \, \pi_{mn}, \, m,n,j=1,\cdots,K\}
        \end{array}$$
\item Initialize $\theta$ randomly.
\item do while none of the parameters of $\theta$ changes
\begin{enumerate}
\item $\alpha_{p+1}(h)=\rho_h \Phi(\frac{y_{p+1}-\mathbf{Y}_{p}^\prime A_h}{\sigma_h})$
\item $\beta_T(h) =1$
\item For t=1 to T do
    \begin{itemize}
        \item $\alpha_{t+1}(h)  =  \sum_m \pi_{m,h} \alpha_t(m)  \Phi(\frac{y_{t+1}-\mathbf{Y}_{t}^\prime A_h}{\sigma_h})$
        \item $\beta_{T-t}(h)  =  \sum_{j=1}^K  \pi_{h,j} \beta_{T-t+1}(j) \Phi(\frac{y_{t+1}-\mathbf{Y}_{t-1}^\prime A_j}{\sigma_j})$
    \end{itemize}
\item $F(Y_{p+1}^T|Y_{1}^p)=\sum_{j=1}^K \alpha_T(j)$
\item For t=1 to T
    \begin{itemize}
        \item $P(Z_t=h|Y_1,\cdots,Y_T)  =   \frac{\alpha_{t}(h)\beta_t(h)}{F(Y_{p+1}^T|Y_{1}^p)}$
        \item $P(Z_t=j,Z_{t-1}=i|Y_1,\cdots,Y_T)  =   \frac{ \pi_{ij} \beta_t(j)  \alpha_{t-1}(i) }{F(Y_{p+1}^T|Y_{1}^p)}\Phi(\frac{y_{t}-\mathbf{Y}_{t-1}^\prime A_j}{\sigma_j})$
    \end{itemize}
\item  and $$\mathbf{P}_k= diag(P(Z_{p+1}=k|y_1...y_T),\cdots,P(Z_{T}=k|y_1...y_T)).$$

\item set the maximum likelihood estimate as
    \begin{itemize}
        \item $\hat{A}_k = (\widetilde{\mathbf{Y}}\mathbf{P}_k\widetilde{\mathbf{Y}}^\prime)^{-1}\widetilde{\mathbf{Y}}\mathbf{P}_k\bar{\mathbf{Y}}$
        \item $\hat{\sigma}_k^2 = \frac{
                    \bar{\mathbf{Y}}^\prime \mathbf{P}_k(\mathbf{I}-\widetilde{\mathbf{Y}}^\prime
                    (\widetilde{\mathbf{Y}}\mathbf{P}_k\widetilde{\mathbf{Y}}^\prime)^{-1})
                    \bar{\mathbf{Y}}\widetilde{\mathbf{Y}}\mathbf{P}_k
                    -
                    2 \bar{\mathbf{Y}}^\prime \mathbf{P}_k \widetilde{\mathbf{Y}}^\prime
                    (\widetilde{\mathbf{Y}}\mathbf{P}_k\widetilde{\mathbf{Y}}^\prime)^{-1}\widetilde{\mathbf{Y}}\mathbf{P}_k\bar{\mathbf{Y}}
                    }
                 {tr(\mathbf{P}_k)}$
        \item $\hat{\pi}_{j,i} = \frac{\sum_{t=p+2}^T P(Z_t=i,Z_{t-1}=j|y_1,\cdots,y_T)}{tr(\mathbf{P}_j)}$
        \item $\hat{\rho}_j = \frac{\sum_{t=P+1}^T P(Z_t=j|Y_1,\cdots,Y_T)}{T-P}$
    \end{itemize}
\end{enumerate}
\end{enumerate}

the convergence of training algorithm is issued by the convergence of all expectation maximization algorithms \cite{mclachlan}.

\begin{Rem}
If all rows of the transition probability matrix, $P$ (\ref{Transition Probabilities}), of hidden Markov chain $\{Z_{t}\}$ are estimated to be equal, then $\{Z_{t}\}$ are independent (i.e. $P(z_{t+1}=j|z_{t}=i)=P(z_{t+1}=j)$) and
\BEA
    \alpha_{t+1}(i) & = &  \sum_{j=1}^K P(Z_{t+1}=i|z_t=j)P(z_t=j|y_1,\cdots,y_t)
    \NNL & = & P(z_{t+1}=i) \sum_{j=1}^K P(z_t=j|y_1,\cdots,y_t) = P(z_{t+1}=i)
    \NNL & = & P(z_{t+1}=i) \sum_{j=1}^K \alpha_{t}(j) = P(z_{t+1}=i)
\EEA
which implies that the weighting coefficients of HM-MAR model can be considered to be fix after parameter estimation. Thus HM-MAR model will result in a MAR model automatically without any further parameter adjustment.
\end{Rem}


\end{document}